\documentclass[11pt]{amsart}

\usepackage{amsmath}
\usepackage{fullpage}
\usepackage{xspace}
\usepackage[psamsfonts]{amssymb}
\usepackage[latin1]{inputenc}
\usepackage{graphicx,color}
\usepackage[curve]{xypic} 
\usepackage{hyperref}
\usepackage{graphicx}
\usepackage{amsmath}%
\usepackage{amsthm}%
\usepackage{amscd}
\usepackage{amsfonts}%
\usepackage{amssymb}%
\usepackage{graphicx}

\usepackage{mathrsfs}

\usepackage{tikz}
\usetikzlibrary{matrix,arrows}

\usepackage{tikz-cd}

\newtheorem{theorem}{Theorem}[section]

\newtheorem{conjecture}[theorem]{Conjecture}
\newtheorem{question}[theorem]{Question}
\newtheorem{corollary}[theorem]{Corollary}

\newtheorem{definition}[theorem]{Definition}

\newtheorem{lemma}[theorem]{Lemma}

\theoremstyle{remark}
\newtheorem{remark}[theorem]{Remark}

\numberwithin{equation}{section}

\newcommand{\Acal}{\mathscr{A}}
\newcommand{\Bcal}{\mathscr{B}}

\newcommand{\Ecal}{\mathscr{E}}

\newcommand{\Kcal}{\mathscr{K}}
\newcommand{\Lcal}{\mathscr{L}}

\newcommand{\Ocal}{\mathscr{O}}

\newcommand{\Tcal}{\mathscr{T}}

\newcommand{\Pro}{\mathbb{P}}

\newcommand{\C}{\mathbb{C}}

\newcommand{\Q}{\mathbb{Q}}
\newcommand{\R}{\mathbb{R}}

\newcommand{\A}{\mathbb{A}}
\newcommand{\B}{\mathbb{B}}

\newcommand{\rk}{\mathrm{rank}\,}

\newcommand{\Proj}{\mathrm{Proj}}

\newcommand{\alg}{\mathrm{alg}}

\newcommand{\an}{\mathrm{an}}

\makeatletter
\@namedef{subjclassname@2020}{%
  \textup{2020} Mathematics Subject Classification}
\makeatother

  \DeclareFontFamily{U}{wncy}{}
    \DeclareFontShape{U}{wncy}{m}{n}{<->wncyr10}{}
    \DeclareSymbolFont{mcy}{U}{wncy}{m}{n}
    \DeclareMathSymbol{\Sha}{\mathord}{mcy}{"58}

\begin{document}
\title[Xeric varieties]{Xeric varieties}

\author{Natalia Garcia-Fritz}
\address{ Departamento de Matem\'aticas,
Pontificia Universidad Cat\'olica de Chile.
Facultad de Matem\'aticas,
4860 Av.\ Vicu\~na Mackenna,
Macul, RM, Chile}
\email[N. Garcia-Fritz]{natalia.garcia@uc.cl}%

\author{Hector Pasten}
\address{ Departamento de Matem\'aticas,
Pontificia Universidad Cat\'olica de Chile.
Facultad de Matem\'aticas,
4860 Av.\ Vicu\~na Mackenna,
Macul, RM, Chile}
\email[H. Pasten]{hector.pasten@uc.cl}%

%\thanks{}
\thanks{N.G.-F. was supported by ANID Fondecyt Regular grant 1211004 from Chile. H.P was supported by ANID Fondecyt Regular grant 1230507 from Chile.}
\date{\today}
%\dedicatory{}

\begin{abstract} Let $X$ be a smooth projective variety over a number field $k$. The Green--Griffiths--Lang conjecture relates the question of finiteness of rational points in $X$ to the triviality of rational maps from abelian varieties to $X$ and to complex hyperbolicity. Here we investigate the phenomenon of sparsity of rational points in $X$ ---roughly speaking, when there are very few rational points if counted ordered by height. We are interested in the case when sparsity holds over every finite extension of $k$, in which case we say that the variety is \emph{xeric}. We initiate a systematic study of the relation of this property with the non-existence of rational curves in $X$ as well as with certain notion of $p$-adic hyperbolicity. 
\end{abstract}

\maketitle

%\tableofcontents

%%%%%%%%%%%%%%%%%%%%%%%%%%%%%%%%%%%%%%
%%%%%%%%%%%%%%%%%%%%%%%%%%%%%%%%%%%%%%
%%%%%%%%%%%%%%%%%%%%%%%%%%%%%%%%%%%%%%
%%%%%%%%%%%%%%%%%%%%%%%%%%%%%%%%%%%%%%
%%%%%%%%%%%%%%%%%%%%%%%%%%%%%%%%%%%%%%
%%%%%%%%%%%%%%%%%%%%%%%%%%%%%%%%%%%%%%

\section{Introduction} 
In this article $k$ will always denotes a number field and $k^{\alg}$ an algebraic closure of $k$. 

%%%%%
%%%%%
\subsection{Finiteness of rational points} 

Following Lang, let us recall that a variety $X$ defined over $k$ is \emph{Mordellic} if for every finite extension $L/k$ we have that the set of $L$-rational points $X(L)$ is finite. The following conjecture by Green--Griffiths--Lang (in the analytic aspect by Green--Griffiths \cite{GG} and in the arithmetic aspect by Lang \cite{Lang}) predicts a characterization of Mordellic varieties in terms of algebraic geometry and in terms of complex analysis:

\begin{conjecture} Let $X$ be a smooth projective geometrically irreducible variety over $k$. The following are equivalent:
\begin{itemize}
\item[(i)] $X$ is Mordellic;
\item[(ii)] For every abelian variety $A$ over $k^{\alg}$, all rational maps $f:A\dasharrow X_{k^{\alg}}$ are constant;
\item[(iii)] For every embedding $\sigma:k\to \C$, the complex manifold $X_\sigma^\an$ is (Brody) hyperbolic, that is, every complex holomorphic map $f:\C\to X_\sigma^\an$ is constant.
\end{itemize}
\end{conjecture}

Here, $X_\sigma^\an$ is defined as the complex manifold $(X\otimes_\sigma \C)^{\an}$. 

It is easy to see that either of (i) or (iii) implies (ii). Any other implication is in general a non-trivial problem. For instance, in the case of curves of genus $g\ge 2$ property (iii) holds by a classical theorem of Picard, while the content of (i) is Mordell's conjecture, proved by Faltings \cite{Faltings1}.

All known cases of (positive dimensional) Mordellic varieties basically come from two sources (up to \'etale covers): projective sub-varieties of the moduli space of principally polarized abelian varieties of a given dimension with suitable level structure \cite{Faltings1}, or hyperbolic sub-varieties of abelian varieties \cite{Faltings2,Faltings3}, both due to Faltings. We remark that the second case builds on work of Vojta \cite{Vojta}. 

%%%%%
%%%%%
\subsection{Sparsity of rational points} While finiteness of rational points is an important qualitative aspect of the arithmetic of varieties, there is also the problem of determining whether the rational points of a variety are difficult to find, even if there are infinitely many of them. This is the problem of \emph{sparsity of rational points}. For instance, one can have infinitely many rational points on an elliptic curve $E$ over $k$, but since the (logarithmic) N\'eron--Tate height is quadratic with respect to the group structure of $E$, the complexity of these points grows rather fast. Before continuing with our discussion, let us make some precise definitions on sparsity of rational points.

Let $X/k$ be a projective variety. Given a line sheaf $\Lcal$ on $X$ there is a (multiplicative) height 
$$
H_{\Lcal}:X(k^{\alg})\to \R_{\ge 0}
$$
normalized to $\Q$, well-defined up to a multiplicative factor $\exp(O(1))$. If $\Lcal$ is ample, then $H_{\Lcal}$ is bounded away from $0$ and it satisfies the Northcott property: finiteness of algebraic points of bounded degree and height. 

Let $U\subseteq X$ be a Zariski open set, and let $L/k$ be a finite extension. If the height associated to $\Lcal$ has the Northcott property on $U$ (e.g. if $\Lcal$ is ample), we can define the counting function on the variable $T\ge 1$
$$
N_\Lcal(U,L,T)=\#\{x\in U(L) : H_\Lcal (x)\le T\}.
$$
We say that the set of $L$-rational points $U(L)$ is \emph{sparse} in $U$ if, choosing $\Lcal=\Acal$ an ample line sheaf, for every $\epsilon>0$ we have the estimate
$$
N_\Acal(U,L,T)\ll T^\epsilon
$$
as $T\to \infty$, where $\ll$ is Vinogradov's notation and the implicit constant depends on possibly all parameters except $T$. From now on, we will write $\ll_{\epsilon}$ to indicate that the estimate should hold for every $\epsilon>0$ and the implicit constant can depend on this parameter.

Any two ample heights can be compared up to a fixed positive power, so, the property that $U(L)$ be sparse in $U$ is actually independent of the choice of ample line sheaf $\Acal$.

We now come to the central notion studied in this article. 

\begin{definition} We say that $U$ is \emph{xeric}\footnote{From ancient greek, meaning ``dry''. We remind the reader that in ecology, the biome ``deserts and xeric shrublands'' is characterized by its sparse (or very limited) vegetation. For example one has the \emph{Atacama Desert} in Chile: most of it has no vegetation at all, while some parts such as \emph{Pampa del Tamarugal} have some rather sparse vegetation. In the same fashion, xeric varieties include those that are Mordellic (finitely many rational points), as well as those whose set of rational points is infinite and sparse.} if $U(L)$ is sparse in $U$ for every finite extension $L/k$. 
\end{definition}

Note that every Mordellic variety is xeric, but the converse is not true, e.g. abelian varieties are xeric by the Mordell--Weil theorem and the quadratic behavior of the N\'eron--Tate height attached to a symmetric ample divisor.

For instance, if $X$ is a smooth projective curve of genus $g$, we have that:

\begin{itemize}
\item If $g=0$, then for some finite extension $L/k$ we have $X_L\simeq \Pro^1_L$, so $X$ is not xeric,
\item if $g=1$, then $X$ is an elliptic curve (after a finite extension of $k$), hence it is xeric, and
\item if $g\ge 2$, then $X$ is xeric because in fact it is Mordellic by Faltings' theorem (note that for the purpose of showing that $X$ is xeric one can simply embed $X$ into its jacobian.) 
\end{itemize}

%%%%%
%%%%%
\subsection{The Main Conjecture} As in the case of the Green--Griffiths--Lang conjecture regarding Mordellic varieties, we consider the problem of giving geometric and analytic characterizations of xeric varieties. In this direction we propose the following:

\begin{conjecture}[Main Conjecture]\label{MainConj} Let $X$ be a smooth projective geometrically irreducible variety over $k$. The following are equivalent:
\begin{itemize}
\item[($1$)] $X$ is xeric;
\item[($2$)] $X_{k^\alg}$ does not contain (possibly singular) rational curves; 
\item[($3_\exists$)] For some prime $p$ and some embedding $\sigma:k\to \C_p$, every $p$-adic analytic map $f:\C_p\to X_\sigma^{\an}$ is constant;
\item[($3_\forall$)] For every prime $p$ and each embedding $\sigma:k\to \C_p$, every $p$-adic analytic map $f:\C_p\to X_\sigma^{\an}$ is constant.
\end{itemize}
\end{conjecture}

Here, $\C_p$ denotes a completion of an algebraic closure of $\Q_p$, $X_\sigma$ is the base change $X\otimes_{\sigma}\C_p$, and the upper-script ``$\an$'' denotes the analitification functor (say, in the sense of Berkovich). As the reader might suspect, for a $p$-adic manifold $Y$ we write $f:\C_p\to Y$ as shorthand for $f:(\A^1_{\C_p})^{\an}\to Y$. 
 
It is clear that ($3_\forall$) implies ($3_\exists$), and that either of ($1$) or ($3_\exists$) implies ($2$).

Conjecture \ref{MainConj} is not independent of existing conjectures in the literature. 

First, we recall Manin's rational curve conjecture \cite{Manin}:  if for some open set $U\subseteq X$ and a finite extension $L/k$ we have that $U(L)$ is not sparse, then $U$ contains a non-empty open set of a rational curve defined over $L$. Thus, Manin's rational curve conjecture implies the equivalence between ($1$) and ($2$); note that unlike Manin's conjecture, we do not require the rational curves in ($2$) to be defined over the ground field.

On the other hand, Cherry conjectured \cite{Cherry} that if $Y$ is a smooth projective manifold over $\C_p$ admitting a non-constant $p$-adic analytic map $f:\C_p\to Y$, then $Y$ contains a (possibly singular) rational curve. A specialization argument to descend to $k^\alg$ then shows that Cherry's conjecture implies the equivalence of ($2$), ($3_\exists$), and ($3_\forall$).

Conjecture \ref{MainConj} holds for curves: let $C$ over $k$ be a smooth projective curve of genus $g$, then
\begin{itemize}
\item If $g=0$ then the four items in the Main Conjecture fail;
\item If $g\ge 1$ then $C$ is xeric as we already explained, it is not rational, and ($3_\forall$) holds by Berkovich's theorem for curves (see below in this section).
\end{itemize}

Perhaps the strongest available evidence for the equivalence of ($1$) and ($2$) is the work of McKinnon \cite{McKinnon}, who shows that this equivalence follows from Vojta's Main Conjecture (on proximity functions) and a conjectural part of the Minimal Model Program (namely, that varieties of negative Kodaira dimension are uniruled). In addition, we mention recent work by Brunebarbe--Maculan \cite{BrunebarbeMaculan} (after work by Ellenberg--Lawrence--Venkatesh \cite{ELV}) showing that varieties with large algebraic fundamental group (in the sense of Koll\'ar) are xeric.

Regarding Cherry's conjecture (and the equivalence of ($2$), ($3_\exists$), and ($3_\forall$)) we recall that Cherry \cite{Cherry} proved that if $A$ is an abelian variety over $\C_p$, then every $p$-adic analytic map $f:\C_p\to A$ is constant. This generalizes an earlier result of Berkovich \cite{Berkovich} showing the analogous statement for curves of genus $g\ge 1$. 

It is worth pointing out that the non-existence of non-constant $p$-adic analytic maps to a $p$-adic manifold is not the only reasonable notion of ``$p$-adic hyperbolicity'' that one can consider (see the work by Javanpeykar--Vezzani \cite{JavVez} for a detailed discussion) but it seems to be the correct one for trying to characterize xeric varieties. For notational convenience, let us make the definition by analogy with the complex case (note that this is different from the terminology in \cite{JavVez}):

\begin{definition} Let $Y$ be a $p$-adic manifold over $\C_p$. We say that $Y$ is Brody hyperbolic if every $p$-adic analytic map $f:\C_p\to Y$ is constant.
\end{definition}

We expect that the Main Conjecture \ref{MainConj} will provide a theoretical framework for studying the problem of sparsity of rational points in a systematic way. The goal of this article is to give some evidence for it and to establish some cases where the different items of this conjecture hold.

%%%%%
%%%%%
\subsection{Variant} For a projective variety $X$ over $k$ we let $Z_{rc}(X)\subseteq X$ be the Zariski closure of the union of all possibly singular rational curves in $X_{k^{\alg}}$. Manin's rational curve conjecture \cite{Manin} suggests the following characterization for the largest xeric open set in $X$:

\begin{conjecture}[The largest xeric open set]\label{ConjLargestXeric} Let $X$ be a smooth projective geometrically irreducible variety over $k$. Then $X-Z_{rc}(X)$ is xeric (thus, it is the largest xeric open set in $X$).
\end{conjecture}

It turns out that this conjecture follows from Vojta's conjecture on proximity functions, and a conjecture in the Minimal Model Program, see the work of McKinnon \cite{McKinnon}.

By lack of evidence, we do not make a conjecture regarding the $p$-adic analytic counterpart of the previous conjecture. Nevertheless, it is tempting to ask the following:

\begin{question} Let $p$ be a prime and let $Y$ be a smooth projective variety over $\C_p$. Let $Z_{rc}(Y)$ be the Zariski closure of the union of all possibly singular rational curves in $Y$, and let $Z_{an}(Y)$ be the Zariski closure of the union of the images of all non-constant $p$-adic analytic maps $f:\C_p\to Y$. Note that $Z_{rc}(Y)\subseteq Z_{an}(Y)$. When do we have equality? 
\end{question}

%%%%%
%%%%%
\subsection{Structure of the paper} \label{SecOutline} First, in Section \ref{SecEtale} we show that the Main Conjecture \ref{MainConj} is compatible with \'etale covers. As a first application of this compatibility, in Section \ref{SecNefT} we will prove the Main Conjecture for all varieties with nef tangent bundle. 

Going to the dual situation, in Section \ref{SecCot} we study the Main Conjecture under various positivity assumptions on the cotangent bundle. For instance, we show that if $X$ is a projective $p$-adic manifold with semi-ample cotangent bundle, then every $p$-adic analytic map $f:\C_p\to X$ is constant. 

In Section \ref{SecSurf} we study the existence of a non-empty xeric open set on surfaces of general type having some irregular \'etale cover. One can easily show that such surfaces have at most finitely many rational curves, and thus, a xeric dense Zariski open set is expected from Conjecture \ref{ConjLargestXeric}. (Here we do not rely on the Bombieri--Lang conjecture for varieties of general type.)

In Section \ref{SecDiag} we consider varieties $X$ with the property that the diagonal $\Delta\subseteq X\times X$ has ``positivity tending to $0$ along \'etale covers of $X$'' ---this will be defined in a precise way by means of a certain geometric invariant $\widehat{s}(X,\Acal)$. The main result of this section is that when $\widehat{s}(X,\Acal)$ vanishes, the variety $X$ is xeric. The results in this section build on work by Tanimoto \cite{Tanimoto}.

Finally, in Section \ref{SecBall} we prove Conjecture \ref{MainConj} for compact ball quotients. The arithmetic aspect follows from work of Brunebarbe--Maculan \cite{BrunebarbeMaculan} and we also provide an alternative proof using our results on $\widehat{s}(X,\Acal)$.

%%%%%
%%%%%
\subsection{Some final questions} As the reader will notice, all examples of xeric varieties that we will discuss (either in the literature or proved here) require infinite algebraic fundamental group. 

\begin{question} Can one produce an example of a positive dimensional xeric variety with finite algebraic fundamental group?
\end{question}

In the analytic setting, William Cherry has conjectured that there is no $p$-adic analytic map $f:\C_p\to X$ to a $K3$ surface having Zariski dense image. We would like to propose the following stronger conjecture:

\begin{conjecture} Let $X/\C_p$ be a smooth projective variety. If there is a $p$-adic analytic map $f:\C_p\to X^\an$ with Zariski dense image, then the Kodaira dimension of $X$ is $\kappa(X)=-\infty$.
\end{conjecture}

Related to the previous conjecture, Daniel Litt asked us the following:

\begin{question} Let $X/\C_p$ be a smooth projective variety admitting a $p$-adic analytic map $f:\C_p\to X^\an$ with Zariski dense image. Is $X$ unirational?
\end{question}

Finally, we mention a more open-ended question asked to us by Jordan Ellenberg:

\begin{question} How does this theory of xeric varieties look like if instead of rational points we consider integral points with respect to divisors? Do we expect a geometric and a $p$-adic analytic conjectural characterization?
\end{question}

%%%%%%%%%%%%%%%%%%%%%%%%%%%%%%%%%%%%%%
%%%%%%%%%%%%%%%%%%%%%%%%%%%%%%%%%%%%%%
%%%%%%%%%%%%%%%%%%%%%%%%%%%%%%%%%%%%%%
%%%%%%%%%%%%%%%%%%%%%%%%%%%%%%%%%%%%%%
%%%%%%%%%%%%%%%%%%%%%%%%%%%%%%%%%%%%%%
%%%%%%%%%%%%%%%%%%%%%%%%%%%%%%%%%%%%%%

\section{\'Etale covers} \label{SecEtale}

In this section we study the compatibility of the different items of Conjecture \ref{MainConj} with \'etale covers.

First we consider the case of rational curves. The following lemma directly follows from the fact that $\Pro^1_{\C}$ is simply connected.

\begin{lemma}[Rational curves and \'etale covers]\label{LemmaRatCurvesEtale} Let $f:Y\to X$ a finite \'etale morphism of smooth projective varieties over an algebraically closed field $K$ of characteristic $0$. Then:
\begin{itemize}
\item[(i)] If $D\subseteq Y$ is a rational curve, then $f(D)$ is a rational curve in $X$.
\item[(ii)] If $C\subseteq X$ is a rational curve, then $f^{-1}(C)$ is a union of rational curves in $Y$. 
\end{itemize}
\end{lemma}

Let us now discuss the case of $p$-adic analytic maps.

\begin{lemma}[$p$-adic maps and \'etale covers] Let $p$ be a prime. Let $\pi:Y\to X$ be a finite \'etale morphism of smooth projective varieties over $\C_p$ and let $f:\C_p\to X^{\an}$ be a $p$-adic analytic map. Then there is a $p$-adic analytic map $\widetilde{f}:\C_p\to Y^{\an}$ such that $f=\pi^{\an}\circ \widetilde{f}$.
\end{lemma}
\begin{proof} By general properties of the analytification functor, the map $\pi^{\an}:Y^{\an}\to X^{\an}$ is finite \'etale. Let $Z=(\A^1)^{\an}\times_{X^{\an}}Y^{\an}$ and consider the associated maps $g:Z\to Y^{\an}$ and $\nu:Z\to (\A^1)^{\an}$. Since analytic base change preserves the property of being finite \'etale, the map $\nu$ is finite \'etale. As $\C_p$ has characteristic $0$, the affine line $\A^1$ admits no non-trivial finite connected \'etale covers (in the category of varieties), and by the Gabber--L\"utkebohmert theorem (cf. Theorem 3.1 in \cite{Lut}) the same holds in the analytic category for $(\A^1)^{\an}$. Thus, $Z$ consists of finitely many disjoint copies of $(\A^1)^{\an}$ mapping isomorphically to $(\A^1)^{\an}$ via $\nu$, and the result follows by restricting $g$ to any of those components of $Z$.
\end{proof}

We remark that a similar application of \cite{Lut} can be found in \cite{JavVez}.

Let us now discuss the compatibility property of being xeric with morphisms. First we have:

\begin{lemma}[Xericity and morphisms]\label{LemmaXericMorphisms} Let $f:Y\to X$ be a  morphism of projective geometrically irreducible varieties over $k$. Let $U\subseteq X$ be an open set and let $V=f^{-1}(U)$. We assume that $f|_V:V\to U$ is finite.
\begin{itemize}
\item[(i)] If $U$ is xeric, so is $V$;
\item[(ii)] If $f$ is finite \'etale over all of $X$, and if $V$ is xeric, so is $U$.
\end{itemize}
\end{lemma}
\begin{proof} Let $\Acal$ be an ample line sheaf on $X$ and define $\Bcal=f^*\Acal$. We count rational points of bounded height with respect to these sheaves, and we can adjust the normalizations so that $H_\Acal(f(x)) = H_\Bcal(x)$ for all $x\in Y(k^\alg)$.

Let us consider (i). Choose an ample line sheaf $\Bcal'$ such that $\Bcal'\otimes\Bcal^{\vee}$ is ample and adjust heights such that  $H_\Bcal(x)\le H_{\Bcal'}(x)$ for all $x\in Y(k^{\alg})$. Recall that $U$ is xeric. Then for every finite extension $L/k$ we have
$$
N_{\Bcal'} (V,L,T)\le N_{\Bcal} (V,L,T) \le (\deg f) N_\Acal(U,L,T)\ll_\epsilon T^\epsilon \quad \mbox{ for all }\epsilon>0
$$
where $\deg(f)$ is the maximal cardinality of a geometric fibre of $f$ over $U$. This is due to the fact that $f(x)\in X(L)$ for $x\in Y(L)$, and we use $H_\Acal(f(x)) = H_\Bcal(x)\le H_{\Bcal'}(x)$.

Let us now consider (ii). Note that now $\Bcal$ is an ample line sheaf on $Y$ because $f$ is finite. Recall that $V$ is xeric and take $L/k$ finite. As $f$ is finite \'etale, by the  Chevalley--Weil theorem there is a finite extension $K/L$ such that $U(L)\subseteq f(V(K))$ and we find
$$
N_{\Acal} (U,L,T) \le  N_\Bcal(V,K,T)\ll_\epsilon T^{\epsilon}  \quad \mbox{ for all }\epsilon>0.
$$
\end{proof}

The previous lemma can be combined with the next observation in applications:

\begin{lemma}[Xericity and extension of scalars]\label{LemmaBaseChange} Let $X$ be a smooth projective geometrically irreducible variety over $k$. Let $U\subseteq X$ be a Zariski open set. Let $K/k$ be a finite extension. We have that $U$ is xeric if and only if $U_K$ is xeric. 
\end{lemma}

The following more technical result does not immediately relate to the property of being xeric, but it will be used in establishing xericity in some cases. See Lemma 4.1 and Proposition 4.4 in \cite{HarariSkorobogatov} (see also Theorem 8.4.6 in \cite{Poonen}).

\begin{lemma}[Rational points and \'etale covers]\label{LemmaRatPtsEtaleTwists}  Let $f: Y\to X$ be a finite \'etale Galois morphism of projective geometrically irreducible varieties over $k$. There are projective geometrically irreducible varieties $Y_1,...,Y_m$ defined over $k$ along with morphisms $f_i : Y_i\to X$ defined over $k$, such that:
\begin{itemize}
\item[(i)] Each $f_i : Y_i\to X$ is a twist of $f : Y\to X$, and
\item[(ii)] $X(k)=\bigcup_{i=1}^m f_i(Y_i(k))$.
\end{itemize}
\end{lemma}

We will often have, a priori, an \'etale cover only defined after extending scalars to $k^\alg$. Thus, the next lemma will be useful in order to descend the situation to $k$ before using the previous lemma. See (the proof of) Claim 4.1 in \cite{BrunebarbeMaculan}\footnote{The discussion in \url{https://mathoverflow.net/questions/364278} is also informative.}.

\begin{lemma}[Producing Galois \'etale covers over $k$] \label{LemmaEtalek} Let $X$ be a projective geometrically irreducible variety defined over $k$ having $X(k)\ne\emptyset$. Let $\pi: X'\to X_{k^{\alg}}$ be a finite \'etale morphism with $X'$ a projective irreducible variety over $k^{\alg}$. There is a projective geometrically irreducible variety $Y$ over $k$ with a finite \'etale Galois morphism $Y\to X$ defined over $k$ such that the induced map $Y_{k^{\alg}}\to X_{k^{\alg}}$ factors through $\pi:X'\to X_{k^{\alg}}$.
\end{lemma}

%%%%%%%%%%%%%%%%%%%%%%%%%%%%%%%%%%%%%%
%%%%%%%%%%%%%%%%%%%%%%%%%%%%%%%%%%%%%%
%%%%%%%%%%%%%%%%%%%%%%%%%%%%%%%%%%%%%%
%%%%%%%%%%%%%%%%%%%%%%%%%%%%%%%%%%%%%%
%%%%%%%%%%%%%%%%%%%%%%%%%%%%%%%%%%%%%%
%%%%%%%%%%%%%%%%%%%%%%%%%%%%%%%%%%%%%%

\section{Varieties with nef tangent bundle} \label{SecNefT}

The goal of this section is to prove Conjecture \ref{MainConj} when the tangent bundle is nef. Here we recall that for a smooth projective variety $X$ over a field $K$ and for a vector bundle $\Ecal$ on $X$, we say that $\Ecal$ is nef, ample, big, semiample, etc. if the line sheaf $\Ocal_P(1)$ on the projective bundle $P=\Proj_X(\Ecal)$ has the corresponding property.

As a warm-up result, let us consider the case of abelian varieties (already discussed in the introduction):
\begin{lemma}\label{LemmaAbelian} Let $A$ be an abelian variety over $k$. Then $A$ is xeric, $A_{k^\alg}$ contains no rational curves, and for every prime $p$ and every embedding $\sigma:k\to \C_p$ we have that $A_{\sigma}^\an$ is Brody hyperbolic.
\end{lemma}
Indeed, the arithmetic part follows from the Mordell--Weil theorem, while $p$-adic Brody hyperbolicity follows from Cherry's theorem \cite{Cherry}.

The main result of this section is the following:
\begin{theorem}\label{ThmNefTangent} Let $X$ be a smooth projective geometrically irreducible variety over $k$ whose tangent bundle $\Tcal_X$ is nef. Then the following are equivalent:
\begin{itemize}
\item[(i)] $X$ is xeric;
\item[(ii)] $X_{k^{\alg}}$ contains no rational curves;
\item[(iii)] for every prime $p$ and every embedding $\sigma:k\to \C_p$, the $p$-adic manifold $X_\sigma^{\an}$ is Brody hyperbolic.
\end{itemize}
\end{theorem}
We note that these three equivalent properties fail in (positive dimensional) projective spaces and hold in abelian varieties. Both examples have nef tangent bundle.

The case of nef tangent bundle is approachable due to the following important result by Demailly, Peternell, and Schneider \cite{DPS}:

\begin{theorem}\label{ThmDPS} Let $Y$ be a projective variety with nef tangent bundle, defined over an algebraically closed field $K$ of characteristic zero. Then there is a finite \'etale cover $\pi:Z\to Y$ defined over $K$ such that the albanese map $a:Z\to A$ of $Z$ is surjective, smooth, and the fibres are Fano varieties.
\end{theorem}
The result in \cite{DPS} is stated for compact K\"ahler manifolds, but if one specializes to projective varieties defined over an algebraically closed field $K$ of characteristic zero, then $\pi:Z\to Y$ can also be defined over $K$ because it is \'etale.

The Fano varieties mentioned in the theorem can be $0$-dimensional, which is the case when $Z$ is an abelian variety. When they are positive dimensional, the following celebrated result of Mori \cite{Mori} will be crucial for us:

\begin{theorem}\label{ThmMori} Let $Y$ be a smooth projective positive dimensional variety over an algebraically closed field $K$ of characteristic $0$. If $Y$ is Fano, then $Y$ contains rational curves defined over $K$.
\end{theorem}

With these results at hand, we are ready for the main point of this section.
\begin{proof}[Proof of Theorem \ref{ThmNefTangent}] The result is trivial if $X$ is a point, so we can assume that $\dim(X)>0$. Items (i), (ii), and (iii) can be checked after base change to a finite extension of $k$, so, thanks to Theorem \ref{ThmDPS} we may assume that we have a finite \'etale cover $\pi:Z\to Y$ defined over $k$ such that the albanese map $a:Z\to A$ of $Z$ (also defined over $k$) is surjective, smooth, and the fibres are Fano varieties.

First, assume that the fibres of $a$ are $0$-dimensional. Then $Z=A$ is an abelian variety and we apply Lemma \ref{LemmaAbelian} to it. Thanks to to the results in Section \ref{SecEtale}, the same holds for $X$.

Finally, assume that the fibres of $a$ are positive dimensional. Then by Mori's Theorem \ref{ThmMori} there is a non-constant morphism $f:\Pro^1_L\to X_L$ defined over a finite extension $L/k$. It follows that (ii) and (iii) fail for $X$. Furthermore, by Lemma \ref{LemmaXericMorphisms}, item (i) also fails for $X$ because it fails for $\Pro^1_L$.
\end{proof}
%%
%%
%%%%%%%%%%%%%%%%%%%%%%%%%%%%%%%%%%%%%%
%%%%%%%%%%%%%%%%%%%%%%%%%%%%%%%%%%%%%%
%%%%%%%%%%%%%%%%%%%%%%%%%%%%%%%%%%%%%%
%%%%%%%%%%%%%%%%%%%%%%%%%%%%%%%%%%%%%%
%%%%%%%%%%%%%%%%%%%%%%%%%%%%%%%%%%%%%%
%%%%%%%%%%%%%%%%%%%%%%%%%%%%%%%%%%%%%%

\section{Positivity of the cotangent bundle} \label{SecCot}

In the previous section we studied Conjecture \ref{MainConj} for varieties whose tangent bundle is nef. Now we consider the dual situation: when the \emph{cotangent} bundle satisfies some positivity assumption.

Compact complex manifolds with ample cotangent bundle are hyperbolic \cite{Kobayashi}. Thus, if $X$ is a smooth projective variety defined over $k$ with ample cotangent bundle $\Omega^1_X$, Lang's conjecture predicts that it is Mordellic. Unconditionally, Moriwaki \cite{Moriwaki} used results of Faltings \cite{Faltings2,Faltings3} to prove Mordellicity under the assumption that $\Omega^1_X$ be ample and \emph{globally generated}.

Regarding our setting, one expects xericity and $p$-adic Brody hyperbolicity under much weaker positivity conditions for $\Omega^1_X$. First, one has the following well-known observation (see \cite{Litt} for a characteristic-free proof):

\begin{lemma} Let $Y$ be a smooth projective variety over an algebraically closed field $K$, having nef cotangent bundle. Then $Y$ contains no rational curve.
\end{lemma}

Therefore, Conjecture \ref{MainConj} implies the following:

\begin{conjecture}[The case of nef cotangent bundle]\label{ConjNefCot} Let $X$ be a smooth projective variety over $k$ having nef cotangent bundle $\Omega^1_X$. Then $X$ is xeric, and for each prime $p$ and each embedding $\sigma:k\to \C_p$ we have that $X_\sigma^{\an}$ is Brody hyperbolic.
\end{conjecture}

Here we would like to give some unconditional support for this conjecture, to the effect that some weak positivity condition on the cotangent bundle should be enough to imply xericity and $p$-adic Brody byperbolicity. 

In the remainder of this section, $X$ is a smooth projective variety over $k$. Let us denote by $\kappa(X)$ the Kodaira dimension of $X$. We have:

\begin{theorem} If $\Omega^1_X$ is nef, the canonical sheaf $\Kcal_X$ is semiample, and $\kappa(X)$ is either $0$ or $1$, then Conjecture \ref{ConjNefCot} holds for $X$.
\end{theorem}
\begin{proof} By H\"oring's structure theorem \cite{Horing} we have that $X_{k^\alg}$ admits an \'etale cover $Y$ satisfying:
\begin{itemize}
\item $Y$ is an abelian variety if $\kappa(X)=0$, or

\item $Y$ is an abelian fibration over a smooth projective curve $C$ of genus $g\ge 2$ if $\kappa(X)=1$.

\end{itemize}

In the first case we conclude by Lemma \ref{LemmaAbelian}. In the second case, we conclude applying Faltings's theorem \cite{Faltings1} (for xericity) or Berkovich's theorem \cite{Berkovich} (for $p$-adic Brody hyperbolicity) on $C$, and then Lemma \ref{LemmaAbelian} on the fibres.
\end{proof}

\begin{remark}
One can hope that Horing's structure theorem (used in the previous proof) might give a way to attack Conjecture \ref{ConjNefCot} in more generality without the assumption $\kappa(X)=0,1$.
\end{remark}

Concerning xericity, without assumptions on $\kappa(X)$ we have the simple (although general) observation:

\begin{lemma} If $\Omega^1_X$ is globally generated, then $X$ is xeric.
\end{lemma}
\begin{proof} The condition implies that the albanese map $a:X\to A$ is a finite (and unramified) morphism. We conclude by Lemmas \ref{LemmaAbelian} and \ref{LemmaXericMorphisms}.
\end{proof}

Finally, regarding $p$-adic Brody hyperbolicity we have:

\begin{lemma}\label{LemmaPAdicSemiample} Let $Y$ be a $p$-adic projective manifold defined over $\C_p$ satisfying that $\Omega^1_Y$ is semiample; i.e. $\Ocal_P(n)$ is globally generated on $P=\Proj_Y(\Omega^1_Y)$ for some $n\ge 1$. Then $Y$ is Brody hyperbolic.
\end{lemma}
\begin{proof} Suppose that we have a non-constant analytic map $f:\C_p\to Y$ and let $z_0\in \C_p$ be such that $f$ has non-zero differential at $z_0$. By our assumption, there is some symmetric differential $0\ne \omega \in H^0(Y,S^n\Omega^1_Y)$ such that $f^*\omega$ is not the zero (analytic) differential on $(\A^1)^\an$. This contradicts the geometric logarithmic derivative lemma of $p$-adic Nevanlinna theory as proved in Lemma 4.3 of \cite{CherryRu}.
\end{proof}

The relation of these two last lemmas with Conjecture \ref{ConjNefCot} is clear: if the cotangent bundle is globally generated then it is semiample, and if it is semiample then it is nef. We note that $\Omega^1_X$ can be globally generated but not ample, e.g. for abelian varieties. 

%%%%%%%%%%%%%%%%%%%%%%%%%%%%%%%%%%%%%%
%%%%%%%%%%%%%%%%%%%%%%%%%%%%%%%%%%%%%%
%%%%%%%%%%%%%%%%%%%%%%%%%%%%%%%%%%%%%%
%%%%%%%%%%%%%%%%%%%%%%%%%%%%%%%%%%%%%%
%%%%%%%%%%%%%%%%%%%%%%%%%%%%%%%%%%%%%%
%%%%%%%%%%%%%%%%%%%%%%%%%%%%%%%%%%%%%%

\section{Surfaces of general type with positive \'etale iregularity} \label{SecSurf}

If $X$ is a smooth projective variety, its irregularity is $q(X)=\dim H^0(X,\Omega^1_X)$. We also define its \'etale irregularity $\hat{q}(X)$ as the maximal $q(Y)$ where $Y$ varies over finite etale covers of $X$. Clearly $\hat{q}(X)\ge q(X)$, and the inequality can be strict.

In this section we study surfaces $X$ of general type (i.e. with Kodaira dimension $2$, or equivalently, with big canonical sheaf $\Kcal_X$) satisfying that $\hat{q}(X)>0$.  First we note:

\begin{lemma} Let $X$ be a surface of general type defined over an algebraically closed field $K$ of characteristic $0$, having $\hat{q}(X)>0$. Then $X$ contains only finitely many rational curves.
\end{lemma}
\begin{proof} By Lemma \ref{LemmaRatCurvesEtale} we may assume that $q:=q(X)>0$. Let $a:X\to A$ be the Albanese map, where $A$ is an abelian variety of dimension $q$. 

If $a$ is quasifinite then the rational curves of $X$ are contained in the (finitely many) positive dimensional fibres of $a$.

On the other hand, if $C=a(X)$ is a curve, then $C$ has geometric genus $g\ge 1$ as it is contained in $A$, and we see that each rational curve of $X$ is contained in some fibre of $a$. These rational curves are finite in number, for otherwise we would have an infinite bounded family of rational curves covering $X$, which is not possible as $X$ is of general type.
\end{proof}

Therefore, from Conjecture \ref{ConjLargestXeric} we expect that for every surface $X$ over $k$ of general type with positive \'etale irregularity, some dense open set of $X$ is xeric. Of course the Bombieri--Lang conjecture implies that if $X$ is a surface of general type (no assumption on irregularity), then some dense open set of $X$ is Mordellic, but this is wide open (unless $X$ is a sub-variety of an abelian variety \cite{Faltings2, Faltings3}) while, as we will see, in the xeric setting one can get more information.

We separate our study in two cases: $\hat{q}(X)\ge 2$ and $\hat{q}(X)=1$. In the first case we have:

\begin{theorem} Let $X$ be a smooth projective surface of general type defined over $k$ having $\hat{q}(X)\ge 2$. Then there is a properly contained Zariski closed set $Z\subseteq X$ such that $U=X-Z$ is xeric.
\end{theorem}
\begin{proof} By Lemma \ref{LemmaXericMorphisms} we may assume that $q:=q(X)\ge 2$. We can also assume that $X(k)\ne \emptyset$ (see Lemma \ref{LemmaBaseChange}).

Let $a:X\to A$ be the Albanese map, where $A$ is an abelian variety of dimension $q$ defined over $k$. If $a$ is quasifinite we may take $Z$ as the union of its positive-dimensional fibres, because $A$ is xeric (Lemma \ref{LemmaAbelian}) and we apply Lemma \ref{LemmaXericMorphisms}. 

Thus, let us assume that $C:= a(X)$ is a curve. As $X$ is of general type it cannot be covered by a bounded family of curves of genus $0$ or $1$, so, only finitely many fibres of $a$ contain such curves. Let $Z$ be the union of these fibres and take $U=X-Z$. 

Let $L/k$ be a finite extension. The curve $C$ spans $A$ and $\dim A=q\ge 2$, so $C$ has geometric genus $g\ge 2$. By Faltings's theorem $C(L)$ is finite and $U(L)\subseteq D(L)$ where $D$ is a finite union of curves of genus $g\ge 2$ on $X$. Thus, $U(L)$ is finite.
\end{proof}

The case $\hat{q}(X)=1$ is more difficult and we can only provide a conditional result. To state it, let us formulate a conjectural bound for the Mordell--Weil rank of abelian varieties. Here, $N_A$ stands for the norm of the conductor ideal of an abelian variety $A$ over a given number field.

\begin{conjecture} \label{ConjMWrk} Let $g\ge 1$ be an integer and let $\epsilon>0$. There is a constant $c=c(g,k,\epsilon)$ such that for all abelian varieties $A$ over $k$ of dimension $g$ we have
$$
\rk A(k) \le \epsilon \cdot\left(c + \log N_A\right).
$$
\end{conjecture}

Although we have not found a precise formulation of this exact conjecture in the literature, we can mention the following motivation:

\begin{itemize}
\item By descent bounds \cite{OoeTop}, the bound $\rk A(k)\ll_{g,k} \log N_A$ is known unconditionally. Thus, any asymptotic improvement on this bound implies the conjecture.

\item It is a folklore question whether there is a uniform bound $c(g,k)$ depending only on the dimension $g$ and the number field $k$ such that for all abelian varieties $A$ over $k$ with $\dim A=g$ one has $\rk A(k)\le c(g,k)$. The case $g=1$ has received considerable attention, see for instance \cite{PPVW} and the references therein for a discussion on this uniform boundedness conjecture, as well as for a survey on the proved and conjectural bounds for the ranks of elliptic curves. There the reader will find explicit references to (stronger forms of) Conjecture \ref{ConjMWrk} for elliptic curves.
\end{itemize}

For the case $\hat{q}(X)=1$ we prove:

\begin{theorem}\label{Thmq1} Assume Conjecture \ref{ConjMWrk}. Let $X$ be a smooth projective surface of general type defined over $k$ having $\hat{q}(X)=1$. Then there is a properly contained Zariski closed set $Z\subseteq X$ such that $U=X-Z$ is xeric.
\end{theorem}

The key to this result is the following theorem by Dimitrov--Gao--Habegger \cite{DGH}:

\begin{theorem}\label{ThmDGH} Let $g\ge 2$. There is a constant $b=b(g,k)>1$ depending only on $g$ and $k$ such that for every smooth projective geometrically irreducible curve $C$ over $k$ of genus $g$, we have
$$
\#C(k) \le b^{1+r(C/k)}
$$
where $r(C/k)$ is the Mordell--Weil rank of the jacobian $J_C$ of $C$ over $k$.
\end{theorem}

\begin{proof}[Proof of Theorem \ref{Thmq1}] By Lemma \ref{LemmaXericMorphisms} we may assume that $q(X)=1$. The Albanese variety of $X$ is an elliptic curve $E$ and (possibly after replacing $k$ by a finite extension) we obtain the Albanese map $f:X\to E$ over $k$. It is surjective with smooth generic fibre (see Exercise 10.40(b) of \cite{GW}). 

The fibres of $f$ are geometrically connected: for otherwise, by applying Stein's factorization theorem (possibly after base change to a finite extension of $k$) we would have a non-trivial factorization $X\to B\to E$ where $B$ is a curve, which necessarily is an elliptic curve as $q=1$. This contradicts the universal property of the Albanese map.

Thus, the generic fibre of $f$ is a smooth projective geometrically irreducible curve of genus $g\ge 2$ ($X$ is of general type, so, it cannot be covered by a bounded family of genus $0$ or $1$ curves).

We can choose a relatively ample line sheaf $\Lcal$ for $f$ on $X$, as well as an ample line sheaf $\Bcal$ on $E$, such that $\Acal=\Lcal\otimes f^*\Bcal$ is ample on $X$. %STACKS PROJECT Lemma 29.37.7.
We adjust normalizations such that we have the following relations of heights for $x\in X(k^{\alg})$:
$$
H_{f^*\Bcal} (x) = H_\Bcal (f(x))\quad\mbox{ and }\quad H_{\Acal}(x) = H_\Lcal(x)\cdot H_{f^*\Bcal} (x)
$$
and furthermore, each of these heights is bounded below by $1$.

Let $Z$ be the union of the bad fibres of $f:X\to E$ and $U=X-Z$. We fix a finite extension $L/k$. We have
$$
\begin{aligned}
N_\Acal(U,L,T)&=\#\{x\in U(L) : H_\Lcal(x)\cdot H_{\Bcal} (f(x))\le T\}\\
&\le \sum_{\substack{y\in E(L)\\ H_{\Bcal}(y)\le T}} \# X_y(L)\quad \mbox{ because }H_{\Lcal}(x)\ge 1,
\end{aligned}
$$ 
where each $X_y$ is a smooth geometrically irreducible projective curve over $L$ of genus $g$. By Theorem \ref{ThmDGH}
$$
\# X_y(L) \le b^{1 + r(X_y/L)} 
$$
with $b=b(g,L)$ as in the theorem. Assuming Conjecture \ref{ConjMWrk}, for every $\epsilon>0$ we have
$$
r(X_y/L) \le \epsilon\cdot ( c + \log N_{J_y})
$$
for certain constant $c$ independent of $y$, where $J_y$ is the jacobian of $X_y$. We observe that $N_{J_y}$ can be bounded using the norm of places of bad reduction of $X_y$ with bounded exponents, which in turn can be bounded by a power of the (multiplicative) height of $y$ relative to the divisor on $E$ supporting the bad fibres of $f:X\to E$. Thus one has
$$
\log N_A \le c'\cdot (1+ \log H_\Bcal (y))
$$
where the constant $c'$ is independent of $y$ and $\epsilon$. Putting all together, we find
$$
\begin{aligned}
N_\Acal(U,L,T) &\ll_{\epsilon} \sum_{\substack{y\in E(L)\\ H_{\Bcal}(y)\le T}} H_\Bcal (y)^{c' \epsilon} \\ 
&\le T^{c' \epsilon} \cdot N_{\Bcal}(E,L,T) \ll_\epsilon T^{(c'+1)\epsilon}
\end{aligned}
$$
where the last bound holds because $E$ is xeric, see Lemma \ref{LemmaAbelian}. This proves that $U$ is xeric.
\end{proof}

%%%%%%%%%%%%%%%%%%%%%%%%%%%%%%%%%%%%%%
%%%%%%%%%%%%%%%%%%%%%%%%%%%%%%%%%%%%%%
%%%%%%%%%%%%%%%%%%%%%%%%%%%%%%%%%%%%%%
%%%%%%%%%%%%%%%%%%%%%%%%%%%%%%%%%%%%%%
%%%%%%%%%%%%%%%%%%%%%%%%%%%%%%%%%%%%%%
%%%%%%%%%%%%%%%%%%%%%%%%%%%%%%%%%%%%%%

\section{Positivity of the diagonal} \label{SecDiag}
 
 The goal of this section is to introduce a new numerical invariant $\hat{s}(X,\Acal)\ge 0$ ($X$ a variety defined over $k$ as always, and $\Acal$ an ample line sheaf on $X$) and to prove that when it vanishes, the variety $X$ is xeric. This new invariant is, in some sense, a limit measure of positivity of the diagonal of $X\times X$ as we vary over suitable \'etale covers of $X$.

For a smooth projective variety geometrically irreducible $M$ over a field $K$ and an ample line sheaf $\Acal$ on $M$, we define $\Acal_j=\pi_j^*\Acal$ for $j=1,2$, where $\pi_j:M\times M\to M$ is the $j$-th projection. The diagonal of $M$ is $\Delta_M\subseteq M\times M$ and we let $b: B_M\to M\times M$ be the blow-up along $\Delta_M$, with exceptional divisor $E_M$. The $s$-invariant (cf. Section 5.4 in \cite{Lazarsfeld}) of the diagonal relative to $\Acal$ is
$$
s(M,\Acal)=\inf \{t\in\Q : t\cdot b^*(\Acal_1\otimes \Acal_2) -E\mbox{ is nef on }B_M \}
$$
where $t\cdot b^*(\Acal_1\otimes \Acal_2)$ is seen as a $\Q$-divisor class on $B_M$. It is a measure of the positivity of $\Delta_M$ relative to $\Acal$. We note that it remains the same under base change to a field extension of $K$.

The following result is due to Tanimoto \cite{Tanimoto}.
\begin{theorem}[Tanimoto]\label{ThmTanimoto} Let $X$ be a smooth projective geometrically irreducible variety over $k$ of dimension $n$. Let $\Acal$ be an ample line sheaf on $X$. Let $\epsilon>0$. Then, as $T\to\infty$ we have
$$
N_\Acal(X,k,T)\ll_\epsilon T^{2[k:\Q]ns+\epsilon}
$$
where $s=s(X,\Acal)$ and the implicit constant is independent of $T$.
\end{theorem}

Note that Tanimoto considers heights normalized to $k$, while our heights are normalized to $\Q$ which explains the factor $[k:\Q]$ in the exponent.

For $M$ and $\Acal$ over a field $K$ as before, let us define another invariant. We let $\hat{s}(M,\Acal)$ be the infimum of the numbers $s(M',f^*\Acal)$ where $f:M'\to M$ varies over all finite  \'etale Galois morphisms defined over $K$ with $M'$ geometrically irreducible. We obtain the following improvement of Tanimoto's thoerem:

\begin{theorem}\label{ThmDiagonalEtale} Let $X$ be a smooth projective geometrically irreducible variety over $k$ of dimension $n$. Let $\Acal$ be an ample line sheaf on $X$. Let $L/k$ be a finite extension and $\epsilon>0$. Then, as $T\to\infty$ we have
$$
N_\Acal(X,L,T)\ll_\epsilon T^{2[L:\Q]n\hat{s}+\epsilon}
$$
where $\hat{s}=\hat{s}(X, \Acal)$ and the implicit constant is independent of $T$. 
\end{theorem}

\begin{corollary} \label{CorosetaleXeric} Let $X$ be a smooth projective geometrically irreducible variety over $k$. If for some ample line sheaf $\Acal$ on $X$ we have that $\hat{s}(X, \Acal)=0$, then $X$ is xeric. 
\end{corollary}

The relevance of replacing $s(X, \Acal)$ by $\hat{s}(X, \Acal)$ is that the former never vanishes, while $\hat{s}(X, \Acal)$ can actually vanish. For instance, taking the multiplication-by-$m$ maps one sees that if $X$ is an abelian variety and $\Acal$ is ample and symmetric, then $\hat{s}(X, \Acal)=0$. Another example where $\hat{s}(X, \Acal)$ vanishes is given by compact ball quotients; this will be analyzed in detail in the next section.

\begin{proof}[Proof of Theorem \ref{ThmDiagonalEtale}] Let $f:Y\to X$ be a finite \'etale Galois morphism over $k$ with $Y$ geometrically connected. It suffices to prove that 
$$
N_\Acal(X,k,T)\ll_\epsilon T^{2[k:\Q]ns_Y+\epsilon}
$$
where $s_Y=s(Y,f^*\Acal)$. The result is then obtained by  choosing $f:Y\to X$ such that $s_Y$ approaches $\hat{s}$, and replacing $k$ by $L$ and $f:Y\to X$ by its base change to $L$; this suffices as the $s$-invariant remains the same under such a base change.

We apply Lemma \ref{LemmaRatPtsEtaleTwists} to get morphisms $f_i:Y_i\to X$ for $i=1,...,m$ (for some finite $m$) defined over $k$, that are twists of $f:Y\to X$, and such that 
$$
X(k) = \bigcup_{i=1}^m f_i(Y_i(k)).
$$

Let $\Lcal_i=f_i^*\Acal$. Then $s_Y=s(Y,f^*\Acal)=s(Y_i,\Lcal_i)$ for each $i$ (we compare them after base change to $k^{\alg}$, where all the $f_i$ become $f$). By functoriality of heights, we then have
$$
N_\Acal(X,k,T)\ll \sum_{i=1}^m N_{\Lcal_i}(Y_i,k,T) 
$$
while Theorem \ref{ThmTanimoto} gives
$$
N_{\Lcal_i}(Y_i,k,T) \ll_\epsilon T^{2[k:\Q]ns_Y+\epsilon}.
$$
The result follows.
\end{proof}

%%%%%%%%%%%%%%%%%%%%%%%%%%%%%%%%%%%%%%
%%%%%%%%%%%%%%%%%%%%%%%%%%%%%%%%%%%%%%
%%%%%%%%%%%%%%%%%%%%%%%%%%%%%%%%%%%%%%
%%%%%%%%%%%%%%%%%%%%%%%%%%%%%%%%%%%%%%
%%%%%%%%%%%%%%%%%%%%%%%%%%%%%%%%%%%%%%
%%%%%%%%%%%%%%%%%%%%%%%%%%%%%%%%%%%%%%

\section{Compact ball quotients} \label{SecBall}

In this section we prove the Main Conjecture \ref{MainConj} for varieties $X$ over $k$ whose complex analytic space (for some embedding of $k$ into $\C$) is a compact ball quotient.

A compact complex manifold $M$ of dimension $n$ will be called \emph{compact ball quotient} if there is a co-compact torsion-free lattice $\Gamma$ in the holomorphic automorphisms group of the $n$-dimensional open unit ball $\B_n\subseteq \C^n$, such that $M$ is biholomorphic to $\B_n/\Gamma$. 

The main result of this section is:

\begin{theorem}[The Main Conjecture holds for compact ball quotients]\label{ThmBall} Let $X$ be a smooth projective geometrically irreducible variety over $k$ such that for some complex embedding $\tau:k\to \C$ the complex manifold $X_\tau^\an$ is a compact ball quotient. Then we have:
\begin{itemize}
\item[(i)] $X$ is xeric,
\item[(ii)] $X_{k^\alg}$ contains no rational curves, and
\item[(iii)] for all primes $p$ and $p$-adic embeddings $\sigma: k\to \C_p$, the $p$-adic manifold $X_\sigma^{\an}$ is Brody hyperbolic.
\end{itemize}
\end{theorem}

Before proving this result we need some preliminaries on compact ball quotients.

Let $M$ and $\Gamma$ be as before. The quotient map $\pi_\Gamma: \B_n\to M=\B_n/\Gamma$ is unramified; in fact, it is the universal covering map. It follows that every finite \'etale cover of $M$ is again a compact ball quotient. Furthermore, as $\B_n$ is hyperbolic so is $M$, but more is true:
\begin{lemma}\label{LemmaBallCot} The cotangent bundle of $M$ is ample, and so is the canonical sheaf $\Kcal_M$.
\end{lemma} 
This result is well-known; it follows from endowing $\Omega^1_M$ with a suitable metric coming from the uniformization $\pi_\Gamma$.

Using $\pi_\Gamma$ one gives $M$ the standard hyperbolic metric which allows one to define its injectivity radius $\rho_M>0$. As $\Gamma$ is residually finite, one has:

\begin{lemma}\label{LemmaBigRho} If $M$ is a compact ball quotient, then there are finite \'etale covers of $M$ with arbitrarily large injectivity radius.
\end{lemma}

As a consequence of the last lemma and results of Hwang and To \cite{HwangTo} we get:

\begin{corollary}\label{CoroBallsetale} Let $X$ be as in Theorem \ref{ThmBall} with $X(k)\ne\emptyset$. Then $\hat{s}(X,\Kcal_X)=0$ (note that $\Kcal_X$ is ample by Lemma \ref{LemmaBallCot}).
\end{corollary}
\begin{proof} Let $M$ be a compact ball quotient of dimension $n$. By Proposition 3 of \cite{HwangTo} the Seshadri constant of the diagonal of $M\times M$ with respect to $\Kcal_M$, which we denote by $\varepsilon(\Kcal_M,\Delta_M)$, satisfies
$$
\varepsilon(\Kcal_M,\Delta_M) \ge 2(n+1) \sinh^2(\rho_M/4).
$$

Let $R$ be any large positive number. By Lemma \ref{LemmaBigRho} there is some finite \'etale cover $N\to X_\tau^\an$  with injectivity radius $\rho_N\ge R$. As the map is \'etale it descends to $k^{\alg}$ and applying Lemma \ref{LemmaEtalek} we find a finite \'etale Galois cover $f:Y\to X$ defined over $k$ which factors through $N$ after base change. As it factors through $N$ we have $\rho_Y\ge \rho_N\ge R$ where $\rho_Y$ is the injectivity radius of the compact ball quotient $Y_\tau^{\an}$.

From its definition, one sees that the Seshadri constant is the reciprocal of the $s$-invariant (see the discussion in Section 5.4 of \cite{Lazarsfeld}). So 
$$
s(Y,\Kcal_{Y}) \le \frac{1}{2(n+1) \sinh^2(R/4)}
$$
with $n=\dim Y=\dim X$. Note that $\Kcal_Y = f^* \Kcal_X$ because $f$ is \'etale. So
$$
\hat{s}(X,\Kcal_X) \le \frac{1}{2(n+1) \sinh^2(R/4)}.
$$
As $R$ can be taken as large as we want, the result follows.
\end{proof}

Finally we have:

\begin{proof}[Proof of Theorem \ref{ThmBall}] We prove each item separately:

\emph{Proof of (i)}. As promised in Section \ref{SecOutline}, we give two different proofs. 

By work of Brunebarbe and Maculan \cite{BrunebarbeMaculan} mentioned in the introduction, smooth projective geometrically irreducible varieties over $k$ with large algebraic fundamental group in the sense of Koll\'ar are xeric. As explained in the same reference, this condition is satisfied by compact ball quotients.

Alternatively: Corollary \ref{CoroBallsetale} gives $\hat{s}(X,\Kcal_X)=0$ (possibly after a finite extension of $k$) and we apply Corollary \ref{CorosetaleXeric} to conclude that $X$ is xeric.

\emph{Proof of (ii)}. As mentioned before, compact ball quotients are (complex) hyperbolic. A rational curve would give a non-constant holomorphic map $\C\to X_\tau^\an$, which is not possible.

\emph{Proof of (iii)}. By Lemma \ref{LemmaBallCot}, we have that $\Omega^1_X$ is ample. Hence it is semiample and we conclude by Lemma \ref{LemmaPAdicSemiample} after base change to $\C_p$ using any embedding $k\to\C_p$.
\end{proof}

%%%%%%%%%%%%%%%%

%%%%%%%%%%%%%%%%%%%%%%%%%%%%%%%%%%%%%%
%%%%%%%%%%%%%%%%%%%%%%%%%%%%%%%%%%%%%%
%%%%%%%%%%%%%%%%%%%%%%%%%%%%%%%%%%%%%%
%%%%%%%%%%%%%%%%%%%%%%%%%%%%%%%%%%%%%%
%%%%%%%%%%%%%%%%%%%%%%%%%%%%%%%%%%%%%%
%%%%%%%%%%%%%%%%%%%%%%%%%%%%%%%%%%%%%%
%%%%%%%%%%%%%%%%%%%%%%%%%%%%%%%%%%%%%%
%%%%%%%%%%%%%%%%%%%%%%%%%%%%%%%%%%%%%%
%%%%%%%%%%%%%%%%%%%%%%%%%%%%%%%%%%%%%%
%%%%%%%%%%%%%%%%%%%%%%%%%%%%%%%%%%%%%%
%%%%%%%%%%%%%%%%%%%%%%%%%%%%%%%%%%%%%%
%%%%%%%%%%%%%%%%%%%%%%%%%%%%%%%%%%%%%%

\section{Acknowledgments}

N.G.-F. was supported by ANID Fondecyt Regular grant 1211004 from Chile. H.P. was supported by ANID Fondecyt Regular grant 1230507 from Chile. This project has been presented in various conferences and seminars, and we benefited from several questions from the audience. In particular we thank Jordan Ellenberg, Daniel Litt, and Ananth Shankar for useful comments.

%%%%%%%%%%%%%%%%%%%%%%%%%%%%%%%%%%%%%%
%%%%%%%%%%%%%%%%%%%%%%%%%%%%%%%%%%%%%%
%%%%%%%%%%%%%%%%%%%%%%%%%%%%%%%%%%%%%%

\end{document}